\documentclass[12pt]{amsart}
\usepackage{amssymb}
\usepackage{graphics}
\usepackage{latexsym}
\usepackage{amsmath}
\usepackage{amssymb,amsthm,amsfonts}
\usepackage{amscd}
\usepackage[arrow, matrix, curve]{xy}
\usepackage{syntonly}
\usepackage{tikz}
\usepackage{tkz-euclide,pgfplots}

\title[Chern number inequalities]{On the Chern number inequalities satisfied by all smooth complete intersection threefolds with ample canonical class}
\author[Mao Sheng]{Mao Sheng}
\author[Jinxing Xu]{Jinxing Xu}
\author[Mingwei Zhang]{Mingwei Zhang}
\email{msheng@ustc.edu.cn}
\email{xujx02@ustc.edu.cn}
\email{zmw309@mail.ustc.edu.cn}
\address{School of Mathematical Sciences,
University of Science and Technology of China, Hefei, 230026, China}

\begin{document}
\theoremstyle{plain}
\newtheorem{thm}{Theorem}[section]
\newtheorem{theorem}[thm]{Theorem}
\newtheorem{lemma}[thm]{Lemma}
\newtheorem{corollary}[thm]{Corollary}
\newtheorem{proposition}[thm]{Proposition}
\newtheorem{addendum}[thm]{Addendum}
\newtheorem{variant}[thm]{Variant}
\theoremstyle{definition}
\newtheorem{construction}[thm]{Construction}
\newtheorem{notations}[thm]{Notations}
\newtheorem{question}[thm]{Question}
\newtheorem{problem}[thm]{Problem}
\newtheorem{remark}[thm]{Remark}
\newtheorem{remarks}[thm]{Remarks}
\newtheorem{definition}[thm]{Definition}
\newtheorem{claim}[thm]{Claim}
\newtheorem{assumption}[thm]{Assumption}
\newtheorem{assumptions}[thm]{Assumptions}
\newtheorem{properties}[thm]{Properties}
\newtheorem{example}[thm]{Example}
\newtheorem{conjecture}[thm]{Conjecture}
\numberwithin{equation}{thm}

\newcommand{\sA}{{\mathcal A}}
\newcommand{\sB}{{\mathcal B}}
\newcommand{\sC}{{\mathcal C}}
\newcommand{\sD}{{\mathcal D}}
\newcommand{\sE}{{\mathcal E}}
\newcommand{\sF}{{\mathcal F}}
\newcommand{\sG}{{\mathcal G}}
\newcommand{\sH}{{\mathcal H}}
\newcommand{\sI}{{\mathcal I}}
\newcommand{\sJ}{{\mathcal J}}
\newcommand{\sK}{{\mathcal K}}
\newcommand{\sL}{{\mathcal L}}
\newcommand{\sM}{{\mathcal M}}
\newcommand{\sN}{{\mathcal N}}
\newcommand{\sO}{{\mathcal O}}
\newcommand{\sP}{{\mathcal P}}
\newcommand{\sQ}{{\mathcal Q}}
\newcommand{\sR}{{\mathcal R}}
\newcommand{\sS}{{\mathcal S}}
\newcommand{\sT}{{\mathcal T}}
\newcommand{\sU}{{\mathcal U}}
\newcommand{\sV}{{\mathcal V}}
\newcommand{\sW}{{\mathcal W}}
\newcommand{\sX}{{\mathcal X}}
\newcommand{\sY}{{\mathcal Y}}
\newcommand{\sZ}{{\mathcal Z}}
\newcommand{\A}{{\mathbb A}}
\newcommand{\B}{{\mathbb B}}
\newcommand{\C}{{\mathbb C}}
\newcommand{\D}{{\mathbb D}}
\newcommand{\E}{{\mathbb E}}
\newcommand{\F}{{\mathbb F}}
\newcommand{\G}{{\mathbb G}}
\newcommand{\HH}{{\mathbb H}}
\newcommand{\I}{{\mathbb I}}
\newcommand{\J}{{\mathbb J}}
\renewcommand{\L}{{\mathbb L}}
\newcommand{\M}{{\mathbb M}}
\newcommand{\N}{{\mathbb N}}
\renewcommand{\P}{{\mathbb P}}
\newcommand{\Q}{{\mathbb Q}}
\newcommand{\R}{{\mathbb R}}
\newcommand{\SSS}{{\mathbb S}}
\newcommand{\T}{{\mathbb T}}
\newcommand{\U}{{\mathbb U}}
\newcommand{\V}{{\mathbb V}}
\newcommand{\W}{{\mathbb W}}
\newcommand{\X}{{\mathbb X}}
\newcommand{\Y}{{\mathbb Y}}
\newcommand{\Z}{{\mathbb Z}}
\newcommand{\id}{{\rm id}}
\newcommand{\rank}{{\rm rank}}
\newcommand{\END}{{\mathbb E}{\rm nd}}
\newcommand{\End}{{\rm End}}
\newcommand{\Hom}{{\rm Hom}}
\newcommand{\Hg}{{\rm Hg}}
\newcommand{\tr}{{\rm tr}}
\newcommand{\Sl}{{\rm Sl}}
\newcommand{\Gl}{{\rm Gl}}
\newcommand{\Cor}{{\rm Cor}}

\def\upcurname{curve1}
\def\belowcurname{curve2}
\def\areaname{\tiny Chang's Region R\small}

\def\coll{black}
\def\colc{black}
\def\colf{white}


\maketitle

\begin{abstract}
We obtain all linear Chern number inequalities satisfied by any smooth complete intersection threefold with ample canonical class.
\end{abstract}



\section{Introduction}

This small note is motivated by finding a new Chern number inequality for a smooth projective threefold $X$ with ample canonical bundle. Let $c_i=c_i(T_{X})$ be its Chern class for $i=1,2,3$. Yau's famous inequality \cite{Yau} in the three dimensional case says that
$$
8c_1c_2 \leq 3c_1^3,
$$
with equality iff $X$ is uniformized by the complex ball. As it contains no $c_3$ term, one may naturally wonder whether there exists a Chern number inequality involving $c_3$. This is possible because of the following result:
\begin{theorem}[Chang-Lopez, Corollary 1.3 \cite{Chang-Lopez}]
The region described by the Chern ratios $(\frac{c_1^3}{c_1c_2},\frac{c_3}{c_1c_2})$ of smooth irreducible threefolds with ample canonical bundle is bounded.
\end{theorem}
However, the result, as well as its proof, does not produce a new Chern number inequality, even for the subclass of smooth complete intersections. Before the discovery of a new method to handle the general case, it is valuable from a scientific standpoint to treat this subclass first by bare hands. This is what we are going to do here.\\

Our method is to determine the convex hull in $\R^2$ generated by Chern ratios $(\frac{c_1^3}{c_1c_2},\frac{c_3}{c_1c_2})$ of all smooth complete intersection threefolds with ample canonical class. Let $n$ be a natural number. A smooth complete intersection (abbreviated as SCI) threefold $X$ in $\P^{n+3}$ is cut out by $n$ hypersurfaces, and a nondegenerate one, i.e., not contained in a hyperplane, by $n$ hypersurfaces of degrees $d_1+1,\cdots,d_n+1$ with $d_i\geq 1$ for $1\leq i\leq n$. The Chern numbers for a smooth $X$ is uniquely determined by the tuple $(d_1,\cdots,d_n)$. Therefore we may use the notation
$$
Q(n;d_1,\cdots,d_n)=(\frac{c_1^3}{c_1c_2}, \frac{c_3}{c_1c_2})\in \mathbb{R}^2
$$
for Chern ratios of $X$. Note that $X$ has ample canonical class if and only if $\sum_{i=1}^{n}d_i\geq 5$. Put
$$
Q=\{Q(n;d_1,\cdots,d_n)| n\geq1, d_i\geq 1, \sum_{i=1}^{n}d_i\geq 5\}\subset \mathbb{R}^2.
$$
Let $P$ be the convex hull of $Q$.
\begin{thm}\label{main thm}
$P$ is a rational polyhedra with infinitely many faces. The corners of $P$ are given by the following points:
$$
\begin{array}{ccc}
  &Q(1;5)=(\frac{1}{16},\frac{43}{8}), &Q(2;2,3)=(\frac{1}{10},\frac{19}{5}), \\ & Q(3;2,3,3)=(\frac{1}{8},\frac{13}{4}),
   & Q(3;2,2,2)=(\frac{1}{3}, \frac{23}{12}),
\end{array}
$$
and $Q(n;1,\cdots,1)=(\frac{2 (-4 + n)^2}{12 - 5 n + n^2}, \frac{-24 + 14 n - 3 n^2 + n^3}{3 (-4 + n) (12 - 5 n + n^2)}) $ with $n\geq 5$.
\end{thm}
\begin{remark}
In another work of M.-C. Chang \cite{Chang}, she described a region $R$ in the plane of Chern ratios such that any rational point in $R$ can be realized by a SCI threefold with ample canonical bundle; Outside $R$ there are infinitely many Chern ratios of smooth complete intersection threefolds but no accumulating points. These two results are related, but do not imply each other. See FIGURE \ref{fig1}.
\end{remark}
We proceed to deduce our main application of the above result. According to the values of their $x$-coordinates, we label the corner points of $P$ as follows:
$$
\begin{array}{ccc}
  p_1=Q(1;5), & p_2=Q(2;2,3), & p_3=Q(3;2,3,3),\\
  p_4=Q(5;1,1,1,1,1), &p_5=Q(3;2,2,2), & p_n=Q(n;1,\cdots,1), n\geq 6.
\end{array}
$$
The sequence of points $\{p_n\}$ converges to the point
$$
p_{\infty}=(2,\frac{1}{3}).
$$

The closure of $P$, denoted by $\bar{P}$, contains the points $p_{n} (n\geq 1), p_{\infty}$ as its corners.

For two distinct points $p,q\in \R^2$, denote the line through $p,q$ by $L_{pq}$, and the line segment connecting $p,q$ by $pq$. Denote the expressions  of lines as follows:
\begin{equation}\notag
\begin{split}
 L_{p_1p_{\infty}}:& \ y=k_0x+b_0,\\
L_{p_mp_{m+1}}: & \ y=k_mx+b_m, m\geq 1.
\end{split}
\end{equation}
The values of $k_m, b_m$ are:

$$
\begin{array}{ccc}
  (k_0,b_0)=(-\frac{242}{93},\frac{515}{93}), &(k_1,b_1)=(-42,8),  &(k_2,b_2)=(-22,6),\\
   (k_3,b_3)=(-14,5), &(k_4,b_4)=(-\frac{9}{2},\frac{41}{12}),&(k_5,b_5)=(-\frac{13}{5},3),
 \end{array}
$$
$$
k_m=\frac{-28 m + m^2 + 4 m^3 - m^4}{(-4 + m) (-3 + m) (-20 - 5 m + 3 m^2)}, \forall \ m\geq 6,
$$
$$
 b_m=\frac{-120 + 254 m + 3 m^2 - 50 m^3 +
 9 m^4}{3 (-4 + m) (-3 + m) (-20 - 5 m + 3 m^2)}, \forall \  m\geq 6.
 $$

 The sequence of lines $L_{p_mp_{m+1}}$ converges to the line
 $$
 L_{\infty}: y=k_{\infty}x+b_{\infty},
 $$
where
$k_{\infty}=-\frac{1}{3}, \ b_{\infty}=1.$

\begin{theorem}\label{main application}
Let $C$ be the convex cone of linear inequalities satisfied by the Chern numbers of each SCI threefold with ample canonical bundle. That is,
\begin{equation}\notag
\begin{split}
C=\{(\lambda_1,\lambda_2,\lambda_3)\in \R^3\mid & \lambda_1c_1^3(X)+\lambda_2c_1(X)c_2(X)+\lambda_3c_3(X)\geq 0,\\
&\textmd{ for any SCI threefold } X \textmd{ with } K_X>0.\}
\end{split}
\end{equation}

Then $C$ is a rational convex cone with edges
$$(-k_0,-b_0,1), \ (k_m,b_m,-1)(m\geq 1), \ (k_{\infty},b_{\infty},-1),$$
where by an edge we mean a one dimensional face.
\end{theorem}

\begin{proof}
Let $\check{C}\subset \R^3$ be the closure of the convex cone generated by the set
\begin{equation}\notag
\{(c_1^3(X), c_1(X)c_2(X), c_3(X))\in \R^3\mid
X \textmd{ is a SCI threefold }
\textmd{with } K_X>0.\}
\end{equation}

Note that if $X$ is a SCI threefold with $K_X>0$, then $c_1(X)c_2(X)<0$. Indeed, Yau's inequality gives us $8c_1(X)c_2(X) \leq 3c_1(X)^3$, and the ampleness of canonical bundle implies the inequality $c_1^3(X)<0$.

Since $c_1(X)c_2(X)<0$, it can be easily seen that
$$
\check{C}=\{(\lambda x,\lambda,\lambda y)\mid \lambda\in \R_{\leq 0},(x,y)\in \bar{P}\}.
$$
By definition, $C$ is the dual cone of  $\check{C}$. By Theorem \ref{main thm}, the codimensional one faces of $\check{C}$ are exactly the hyperplanes in $\R^3$ determined by the vectors $(-k_0,-b_0,1)$, $(k_m,b_m,-1)(m\geq 1)$, $(k_{\infty},b_{\infty},-1)$. From the duality of $C$ and $\check{C}$ we get that $(-k_0,-b_0,1)$, $(k_m,b_m,-1)(m\geq 1)$, $(k_{\infty},b_{\infty},-1)$ are exactly the edges (one-dimensional faces) of $C$.
\end{proof}

\begin{corollary}\label{cor}
If $X$ is a SCI threefold with $K_X>0$, then its Chern numbers satisfy the inequality
$86c_1^3\leq c_3<\frac{c_1^3}{6}$
, with the equality $86c_1^3= c_3$ holds if and only if $X$ is isomorphic to a degree $6$ hypersurface in $\P^4$.

\end{corollary}
\begin{proof}
It can be checked that
$$
\frac{744}{229}(-k_0,-b_0,1)+\frac{515}{229}(k_1,b_1,-1)=(-86,0,1).
$$

By Theorem \ref{main application}, $(-86,0,1)\in C$, hence we have
$$
c_3-86c_1^3\geq 0,
$$
with equality holds iff
$$
(\frac{c_1^3}{c_1c_2},\frac{c_3}{c_1c_2})=L_{p_1p_{\infty}}\cap L_{p_1p_2}=Q(1;5),
$$
which is equivalent to that $X$ is isomorphic to a degree $6$ hypersurface in $\P^4$.

Similarly,
$$
\frac{93}{422}(-k_0,-b_0,1)+\frac{515}{422}(k_\infty,b_{\infty},-1)=(\frac{1}{6},0,-1).
$$

By Theorem \ref{main application}, we have
$$
\frac{1}{6}c_1^3-c_3\geq 0,
$$
with equality holds iff
$$
(\frac{c_1^3}{c_1c_2},\frac{c_3}{c_1c_2})=L_{p_1p_{\infty}}\cap L_{\infty}=p_{\infty}.
$$

Since $p_{\infty}$ is not in $Q$, the inequality $\frac{1}{6}c_1^3-c_3\geq 0$ is in fact strict.

\end{proof}

\begin{remark}
In \cite{Lu-Tan-Zuo}, the authors prove that for a smooth projective threefold $X$ admitting a smooth fibration of minimal surfaces of general type over a curve, it holds that $$c_3(X)\geq \frac{1}{18}c_1^3(X).$$ According to Corollary \ref{cor}, this inequality can never be satisfied for a SCI threefold with ample canonical bundle. As pointed out by Professor Kang Zuo, this is actually explained by the Lefschetz hyperplane theorem. Indeed, the hyperbolicity of the moduli space of minimal surfaces of general type means that the base curve of a smooth fibration is a hyperbolic curve and hence the fundamental group of the total space $X$ is nontrivial, in contrast with the triviality of the fundamental group of a SCI implied by the Lefschetz hyperplane theorem.
\end{remark}

\section{Proof of the Theorem}

Suppose $n\in \N$ is a positive integer, and $d_1,\cdots, d_n \in \R_{\geq 0}$ are nonnegative real numbers. Let $s_j=\sum_{i=1}^{n}d_i^j$, $j\geq 1$. We define
\begin{equation}\notag
\begin{split}
c_1(n;d_1,\cdots, d_n)&:=4-s_1,\\
c_2(n;d_1,\cdots, d_n)&:=\frac{s_1^2+s_2}{2}-3(s_1-2),\\
c_3(n;d_1,\cdots, d_n)&:=-\frac{s_1^3+3s_1s_2+2s_3}{6}+(s_1^2+s_2)-3s_1+4.
\end{split}
\end{equation}
If $d_1=\cdots=d_n=d$, we denote $c_i(n;d_1,\cdots,d_n)$ by $c_i(n;d)$, $i=1,2,3.$

The following result is standard.
\begin{lemma}\label{expression of chern numbers}
Let $X$ be a SCI threefold in $\mathbb{P}^{n+3}$. If $X$ is a complete intersection of hypersurfaces with degrees $d_1+1,\cdots,d_n+1$, and $d_i\geq 1$, $\forall \ 1\leq i\leq n$, then the Chern classes of $X$ are: $c_i(X)=c_i(n;d_1,\cdots,d_n)$, $i=1,2,3$.

\end{lemma}

We divide the proof of Theorem \ref{main thm} into three steps, corresponding to three sections.

Step 1:  In section \ref{section:p1pinfty}, we firstly prove the $x$-coordinate of any point in $Q$ is between the $x$-coordinates of $p_1$ and $p_{\infty}$. Then we prove any point of  $Q$ is below the line $L_{p_1p_{\infty}}$.

Step 2: In section \ref{section:pm}, we prove that any point of  $Q$ is above the line $L_{p_mp_{m+1}}$, $\forall \ m\geq 6$.

Step 3: In section \ref{section:p1p6}, we prove that $\forall \ i=1,\cdots, 5$, if a point $ Q(n;d_1,\cdots, d_n)\in Q$ has $x$-coordinate less than or equal to the $x$-coordinate of $p_6$, \ then  $Q(n;d_1,\cdots, d_n)$ lies above the line segment $p_ip_{i+1}$.

After the three steps above, it is obviously we have finished the proof of Theorem \ref{main thm}.

Ideas of the proof in each step:

In steps 1 and 2, the idea of the proof is the following:

Given a line $L: y=kx+b$ in $\R^2$,  to prove $Q$ is below $L$ is equivalent to verify $\forall \ n,d_i\in \N, \sum_{i=1}^n d_i\geq 5,$
$$
c_3(n;d_1,\cdots,d_n)-kc_1^3(n;d_1,\cdots,d_n)-bc_1(n;d_1,\cdots,d_n)c_2(n;d_1,\cdots,d_n)\geq 0,
$$

and by the following Lemma \ref{reduce to equal case}, it suffices to verify $\forall \ n\in \N, d\in \R, d\geq 1,  nd\in \N, nd\geq 5$,
$$
c_3(n;d)-kc_1^3(n;d)-bc_1(n;d)c_2(n;d)\geq 0.
$$

Similarly, in order to prove $Q$ is above a line $y=kx+b$, it suffices to verify $\forall \ n\in \N, d\in \R, d\geq 1,  nd\in \N, nd\geq 5$,
$$
kc_1^3(n;d)+bc_1(n;d)c_2(n;d)-c_3(n;d)\geq 0.
$$

\begin{lemma}\label{reduce to equal case}
Let $\lambda,\mu,\nu\in \R$ be constants.  $\forall \ m\in \N$, we have
\begin{equation}\notag
\begin{split}
inf\{&\lambda c_1^3(n;d_1,\cdots,d_n)+\mu c_1(n;d_1,\cdots,d_n)c_2(n;d_1,\cdots,d_n)+ \nu c_3(n;d_1,\cdots,d_n)\mid \\
 &n\in \N, d_i\in \R, \sum_{i=1}^{n}d_i= m,  d_i\geq 1, \forall \ i=1,\cdots, n.\} \geq \\
inf\{& \lambda c_1^3(n;d)+\mu c_1(n;d)c_2(n;d)+ \nu c_3(n;d) \mid  n\in \N, d\in \R, nd=m, d\geq 1.  \}
\end{split}
\end{equation}
\end{lemma}
\begin{proof}
This lemma  is a direct consequence  of the following elementary proposition.
\end{proof}

\begin{proposition}
Let $d_1\leq d_2\leq \cdots\leq d_n$ be nonnegative real numbers, $s_j=\sum_{i=1}^{n}d_i^j$, $j=1,2,3$,
 and $\lambda, \mu\in \mathbb{R}$ be constants. For fixed $n$ and $s_1$, there exists a natural number $k\leq n$, such that the function $\lambda s_2+ \mu s_3$ attains its minimal value when  $d_1=\cdots =d_k=0$, and $d_{k+1}=\cdots =d_n= \frac{s_1}{n-k}$.
\end{proposition}

In step 3, we firstly prove that there are only finite points of $Q$ with $x$-coordinates less than or equal to the $x$-coordinate of $p_6$, then we verify case-by-case that if a point of $Q$ has $x$-coordinate less than or equal to the  $x$-coordinate of $p_6$,  it lies above the union of line segments $\cup_{i=1}^{5}p_ip_{i+1}$.

\subsection{}\label{section:p1pinfty}

We first give an estimate of the $x$-coordiantes of points in $Q$.

\begin{lemma}
The $x$-coordinate of any point of $Q$ is between the $x$-coordiantes of $p_1$ and $p_{\infty}$.

\end{lemma}

\begin{proof}
Recall the $x$-coordinates of $p_1$ and $p_{\infty}$ are $\frac{1}{16}$ and $2$ respectively. For any point $Q(n;d_1,\cdots,d_n)$ in $Q$, by Lemma \ref{expression of chern numbers}, its $x$-coordinate is
\begin{displaymath}
\frac{c_1(n;d_1,\cdots,d_n)^3}{c_1(n;d_1,\cdots,d_n)c_2(n;d_1,\cdots,d_n)}=
\frac{2(4-s_1)^2}{s_1^2+s_2-6(s_1-2)}
\end{displaymath}
where $s_j=\sum_{i=1}^{n}d_i$, $j=1,2$.

What we want to prove is equivalent to
\begin{displaymath}
\frac{1}{16}\leq \frac{2(4-s_1)^2}{s_1^2+s_2-6(s_1-2)}\leq 2
\end{displaymath}

Since $s_1\geq 5$, the inequalities above can be verified easily.

\end{proof}

Recall the line $L_{p_1p_{\infty}}$ has the expression  $y=k_0x+b_0$, where $k_0=-\frac{242}{93}, b_0=\frac{515}{93}$. According to the argument before Lemma \ref{reduce to equal case}, in order to prove $Q$ is below $L_{p_1p_{\infty}}$, it suffices to prove the following :
\begin{lemma}
$\forall \ n\in \N, d\in \R, nd \in \N, nd\geq 5, d\geq 1,$
$$
f(n,d):=c_3(n;d)-k_0c_1^3(n;d)-b_0c_1(n;d)c_2(n;d)\geq 0.
$$
\end{lemma}
\begin{proof}
Let $s_1=nd$, we have
$$
f(n,d)=\tilde{f}(s_1,d)=\frac{1}{93} (3500 - 2625 s_1 - 937 d s_1 - 31 d^2 s_1 + 422 s_1^2 + 211 d s_1^2).
$$
Since $\tilde{f}(s_1,d) $ is a quadratic polynomial of $d$ with negative leading term, and $1\leq d\leq s_1$, we have $\tilde{f}(s_1,d)\geq Min\{\tilde{f}(s_1, 1),\tilde{f}(s_1,s_1)\}$. By computations,

$f(s_1,1)= \frac{1}{93}(3500 - 3593 s_1 + 633 s_1^2)$,

$f(s_1,s_1)= \frac{5}{93}(700 - 525 s_1- 103 s_1^2 + 36 s_1^3)$.

It is elementary to verify the above two polynomials of $s_1$ are nonnegative when $s_1\in \mathbb{N}$ and $s_1\geq 5$. Hence $f(n,d)=\tilde{f}(s_1,d)\geq 0, \forall \ n\in \N, d\in \R, nd \in \N, nd\geq 5, d\geq 1.$

\end{proof}

\subsection{}\label{section:pm}
In this section we prove that $Q$ is above the line  $L_{p_mp_{m+1}}$, $\forall \ m\geq 6$.
Recall the line $L_{p_mp_{m+1}}$ has an expression $y=k_mx+b_m$, where
\begin{equation}\notag
\begin{split}
k_m&=\frac{-28 m + m^2 + 4 m^3 - m^4}{(-4 + m) (-3 + m) (-20 - 5 m + 3 m^2)},\\
b_m&=\frac{-120 + 254 m + 3 m^2 - 50 m^3 +
 9 m^4}{3 (-4 + m) (-3 + m) (-20 - 5 m + 3 m^2)}.
\end{split}
\end{equation}

 According to the argument before Lemma \ref{reduce to equal case}, to prove $Q$ is above the line $L_{p_mp_{m+1}}$,  we need to study the nonnegativity of the function
 $$f(m,n,d):=k_mc_1^3(n;d)+b_mc_1(n;d)c_2(n;d)-c_3(n;d).$$
 We have the following lemma.

\begin{lemma}\label{positivity}
$f(m,n,d)\geq 0$, if one of the following conditions holds:
\begin{enumerate}
  \item $m, n\in \N, d\in \R, d\geq 1, nd \in \N, nd\geq 5, m\geq 10$;
  \item $m,n\in \N, d\in \R, d\geq 1, nd \in \N, nd\geq 11, m=6,7,8,9$.
\end{enumerate}
\end{lemma}
\begin{proof}
Let $s_1=nd$. In the new variable $m, s_1, d$, we denote the function $f(m,n,d) $ by $\tilde{f}(m,s_1,d)$, then by the expressions of $k_m,b_m$ and $c_i(n;d)$,
\begin{equation}\notag
\begin{split}
\tilde{f}(m,s_1,d)=&f(m, n, d )\\
=&
k_m(4 - s_1)^3 +
 b_m(4 - s_1)(\frac{s_1^2}{2} -
    3 s_1 +6) +\frac{s_1^3}{6}  - s_1^2  + 3 s_1 -
     4\\
 & + \frac{s_1d^2}{3}  +(\frac{s_1^2}{2}-s_1+ \frac{b_m (4 -s_1 ) s_1}{2} )d.
\end{split}
\end{equation}

Suppose condition $(1)$ holds, since $\tilde{f}(m,s_1,d)$ is a quadratic polynomial of $d$, we have two cases:

Case I:
$$\frac{-(s_1^2/2-s_1+ b_m (4 -s_1 ) s_1/2 )}{2s_1/3}\leq 1.$$

In this case, $\tilde{f}(m,s_1,d)\geq \tilde{f}(m,s_1,1)=f(m,s_1,1)$, and $f(m,s_1,1)\geq 0$ is equivalent to that the point $Q(s_1;1,\cdots,1)$ lies above the line $L_{p_mp_{m+1}}$, which can be easily verified under the condition $s_1\geq 5$.

Case II:
 $$\frac{-(s_1^2/2-s_1+ b_m (4 -s_1 ) s_1/2 )}{2s_1/3}\geq 1.$$

 In this case, $s_1\geq \frac{12b_m-2}{3b_m-3}$, and
\begin{equation}\notag
\tilde{f}(m,s_1,d)\geq \tilde{f}(m,s_1, \frac{-(s_1^2/2-s_1+ b_m (4 -s_1 ) s_1/2 )}{2s_1/3}).
\end{equation}

By  computations, we get
 \begin{equation}\notag
 \begin{split}
g(m,s_1):=& 12 (-4 + m)^2 (-3 + m)^2 (-20 - 5 m + 3 m^2)^2 \cdot \\
&\tilde{f}(m,s_1, \frac{-(s_1^2/2-s_1+ b_m (4 -s_1 ) s_1/2 )}{2s_1/3})
\end{split}
\end{equation}
is a cubic  polynomial of $s_1$ with polynomial coefficients of $m$. We only need to verify the positivity of $g(m,s_1)$.

Again, by  computations, we get that if  $m\in \N, m\geq 10$, then
\begin{equation}\notag
\begin{split}
&g(m,\frac{12b_m-2}{3b_m-3})>0, \ \ \frac{\partial g}{\partial s_1}(m, \frac{12b_m-2}{3b_m-3})>0,\\
&\frac{\partial^2 g}{\partial s_1^2}(m, \frac{12b_m-2}{3b_m-3})>0, \ \ \frac{\partial^3 g}{\partial s_1^3}(m, \frac{12b_m-2}{3b_m-3})>0.
\end{split}
\end{equation}

 Note  $g(m,s_1)$ is a cubic polynomial of $s_1$, the positivity of $g(m,s_1)$ follows from the above computations, and we have verified the conclusion under condition $(1)$.

 Suppose condition $(2)$ holds, by computations, we have
 $$\frac{12b_m-2}{3b_m-3}< 10, \ \forall \ m=6,7,8,9.$$

 This inequality and  condition $(2)$ imply $s_1>\frac{12b_m-2}{3b_m-3} +1>\frac{12b_m-2}{3b_m-3}$. We have
 \begin{equation}\notag
\tilde{f}(m,s_1,d)\geq \tilde{f}(m,s_1, \frac{-(s_1^2/2-s_1+ b_m (4 -s_1 ) s_1/2 )}{2s_1/3}).
\end{equation}

 Again let
 \begin{equation}\notag
 \begin{split}
g(m,s_1):=& 12 (-4 + m)^2 (-3 + m)^2 (-20 - 5 m + 3 m^2)^2 \cdot \\
&\tilde{f}(m,s_1, \frac{-(s_1^2/2-s_1+ b_m (4 -s_1 ) s_1/2 )}{2s_1/3})
\end{split}
\end{equation}
 which is a cubic polynomial of $s_1$. We only need to show the positivity of $g(m,s_1)$. By direct computations,
\begin{equation}\notag
\begin{split}
&g(m,\frac{12b_m-2}{3b_m-3}+1)>0, \ \ \frac{\partial g}{\partial s_1}(m, \frac{12b_m-2}{3b_m-3}+1)>0,\\
&\frac{\partial^2 g}{\partial s_1^2}(m, \frac{12b_m-2}{3b_m-3}+1)>0, \ \ \frac{\partial^3 g}{\partial s_1^3}(m, \frac{12b_m-2}{3b_m-3}+1)>0.
\end{split}
\end{equation}

Since $g(m,s_1)$ is a cubic polynomial of $s_1$, the positivity of $g(m,s_1)$ follows from the above computations. And we have verified the conclusion under condition $(2)$.

\end{proof}

From Lemma \ref{reduce to equal case} and Lemma \ref{positivity}, we get that $Q(n;d_1,\cdots,d_n)\in Q$ lies above the line $L_{p_mp_{m+1}}$, if one of the following conditions holds:
\begin{enumerate}
  \item $m\geq 10$;
  \item $m=6,7,8,9, \sum_{i=1}^{n}d_i\geq 11$.
\end{enumerate}

A case-by-case verification shows that if $5\leq\sum_{i=1}^{n}d_i\leq 10$, then $Q(n;d_1,\cdots,d_n)$ lies above the line  $L_{p_mp_{m+1}}$, $\forall \ m=6,7,8,9$. So we have verified $Q$ lies above the line $L_{p_mp_{m+1}}$, $\forall \ m\geq 6$.

\subsection{}\label{section:p1p6}
In this section, we prove that $\forall \ i=1,\cdots, 5$, if a point $ Q(n;d_1,\cdots, d_n)\in Q$ has $x$-coordinate less than or equal to the $x$-coordinate of $p_6$, then  $Q(n;d_1,\cdots, d_n)$ lies above the line segment $p_ip_{i+1}$.

 Note that the $x$-coordinate of $p_6$ is $\frac{4}{9}$. The following lemma tells us that there are only finite points in $Q$ with $x$-coordinate less than or equal to the $x$-coordinate of $p_6$.

\begin{lemma}
 If $s_1=\sum_{i=1}^{n}d_i\geq 10$, then
\begin{equation}\notag
\frac{c_1(n;d_1,\cdots,d_n)^3}{c_1(n;d_1,\cdots,d_n)c_2(n;d_1,\cdots,d_n)}=\frac{(4-s_1)^2}{(s_1^2+s_2)/2-3(s_1-2)}> \frac{4}{9}.
\end{equation}
\end{lemma}

\begin{proof}
Since $s_1^2\leq s_2$, we have
\begin{equation}\notag
\begin{split}
&s_1\geq 10 \Rightarrow s_1^2-12s_1+24>0 \Rightarrow 5s_1^2-60s_1+120>0\\
&\Rightarrow 7s_1^2-60s_1+120>2s_1^2\Rightarrow 7s_1^2-60s_1+120>2s_2\\
&\Rightarrow \frac{(4-s_1)^2}{(s_1^2+s_2)/2-3(s_1-2)}> \frac{4}{9}.
\end{split}
\end{equation}
\end{proof}

A case-by-case verification shows that the finite points
$\{Q(n,d_1,\cdots,d_n), 5\leq \sum_{i=1}^{n}d_i\leq 9\}$ lie above the lines $\{L_{p_ip_{i+1}}, 1\leq i\leq 5\}$. By these verifications and the above lemma, we have shown that, once the $x$-coordinate of a point $Q(n;d_1,\cdots, d_n)$ is less than or equal to that of $p_6$, it lies above the line  $L_{p_ip_{i+1}}$, $\forall \ 1\leq i\leq 5$. This completes the proof of Theorem \ref{main thm}.

\section*{Acknowledgments}
The authors would like to express warm thanks to Professor Sheng-Li Tan and Professor Kang Zuo for helpful discussions and comments. We would like also to thank Professor Ulf Persson for his comments, especially sharing with us his conjecture on the geography of the Chern invariants of threefolds. This work is partially supported by Wu Wen-Tsun Key Laboratory of Mathematics, University of Science and Technology of China.

\end{document}